\documentclass[10pt,oneside]{amsart}

\usepackage[
paper=a4paper,
headsep=20pt,text={32pc,50pc},centering,includehead
]{geometry}

\usepackage[bookmarks]{hyperref}
\usepackage{amssymb,amsxtra,dsfont}
\usepackage[all]{xy}
\usepackage{pb-diagram,pb-xy}
\usepackage{graphicx,color}
\usepackage{epstopdf}
\usepackage{pinlabel}
\usepackage{float}
\usepackage{array,calc,booktabs}
\usepackage[
  textwidth=1.2in,
  backgroundcolor=yellow,
  bordercolor=orange,
  textsize=small
]{todonotes}

\iftrue
\makeatletter
\def\@maketitle{%
  \normalfont\normalsize
  \@adminfootnotes
  \@mkboth{\@nx\shortauthors}{\@nx\shorttitle}%
  \global\topskip42\p@\relax 
  \@settitle
  \ifx\@empty\authors \else \@setauthors \fi
  \ifx\@empty\@dedicatory
  \else
    \baselineskip22\p@
    \vtop{{\small\itshape\@dedicatory\@@par}%
      \global\dimen@i\prevdepth}\prevdepth\dimen@i
  \fi
  \@setabstract
  \normalsize
  \if@titlepage
    \newpage
  \else
    \dimen@34\p@ \advance\dimen@-\baselineskip
    \vskip\dimen@\relax
  \fi
} 
\def\@settitle{%
  \vspace*{-15pt}
  \begin{flushleft}%
    \LARGE\bfseries
    \strut\@title\strut
  \end{flushleft}%
}
\def\@setauthors{%
  \begingroup
  \def\thanks{\protect\thanks@warning}%
  \trivlist
  \raggedright
  \large \@topsep27\p@\relax
  \advance\@topsep by -\baselineskip
  \item\relax
  \author@andify\authors
  \def\\{\protect\linebreak}%
  \authors
  \ifx\@empty\contribs
  \else
    ,\penalty-3 \space \@setcontribs
    \@closetoccontribs
  \fi
  \normalfont
  \endtrivlist
  \endgroup
}
\def\@setaddresses{\par
  \nobreak \begingroup
  \small\raggedright
  \def\author##1{\nobreak\addvspace\smallskipamount}%
  \def\\{\unskip, \ignorespaces}%
  \interlinepenalty\@M
  \def\address##1##2{\begingroup
    \par\addvspace\bigskipamount\noindent
    \@ifnotempty{##1}{(\ignorespaces##1\unskip) }%
    {\ignorespaces##2}\par\endgroup}%
  \def\curraddr##1##2{\begingroup
    \@ifnotempty{##2}{\nobreak\noindent\curraddrname
      \@ifnotempty{##1}{, \ignorespaces##1\unskip}\/:\space
      ##2\par}\endgroup}%
  \def\email##1##2{\begingroup
    \@ifnotempty{##2}{\nobreak\noindent E-mail address%
      \@ifnotempty{##1}{, \ignorespaces##1\unskip}\/:\space
      \ttfamily##2\par}\endgroup}%
  \def\urladdr##1##2{\begingroup
    \def~{\char`\~}%
    \@ifnotempty{##2}{\nobreak\noindent\urladdrname
      \@ifnotempty{##1}{, \ignorespaces##1\unskip}\/:\space
      \ttfamily##2\par}\endgroup}%
  \addresses
  \endgroup
  \global\let\addresses=\@empty
}
\def\@setabstracta{%
    \ifvoid\abstractbox
  \else
    \skip@17pt \advance\skip@-\lastskip
    \advance\skip@-\baselineskip \vskip\skip@
    \box\abstractbox
    \prevdepth\z@ 
    \vskip-15pt
  \fi
}
\renewenvironment{abstract}{%
  \ifx\maketitle\relax
    \ClassWarning{\@classname}{Abstract should precede
      \protect\maketitle\space in AMS document classes; reported}%
  \fi
  \global\setbox\abstractbox=\vtop \bgroup
    \normalfont\small
    \list{}{\labelwidth\z@
      \leftmargin0pc \rightmargin\leftmargin
      \listparindent\normalparindent \itemindent\z@
      \parsep\z@ \@plus\p@
      
    }%
    \item[\hskip\labelsep\bfseries\abstractname.]%
}{%
  \endlist\egroup
  \ifx\@setabstract\relax \@setabstracta \fi
}

\def\ps@headings{\ps@empty
  \def\@evenhead{%
    \setTrue{runhead}%
    \normalfont\scriptsize
    \rlap{\thepage}\hfill
    \def\thanks{\protect\thanks@warning}%
    \leftmark{}{}}%
  \def\@oddhead{%
    \setTrue{runhead}%
    \normalfont\scriptsize
    \def\thanks{\protect\thanks@warning}%
    \rightmark{}{}\hfill \llap{\thepage}}%
  \let\@mkboth\markboth
}\ps@headings

\def\section{\@startsection{section}{1}%
  \z@{-1.4\linespacing\@plus-.5\linespacing}{.8\linespacing}%
  {\normalfont\bfseries\Large}}
\def\subsection{\@startsection{subsection}{2}%
  \z@{-.8\linespacing\@plus-.3\linespacing}{.5\linespacing\@plus.2\linespacing}%
  {\normalfont\bfseries\large}}
\def\subsubsection{\@startsection{subsubsection}{3}%
  \z@{.7\linespacing\@plus.2\linespacing}{-1.5ex}%
  {\normalfont\bfseries}}
\def\@secnumfont{\bfseries}

\renewcommand\contentsnamefont{\bfseries}
\def\@starttoc#1#2{\begingroup
  \setTrue{#1}%
  \par\removelastskip\vskip\z@skip
  \@startsection{}\@M\z@{\linespacing\@plus\linespacing}%
    {.5\linespacing}{
      \contentsnamefont}{#2}%
  \ifx\contentsname#2%
  \else \addcontentsline{toc}{section}{#2}\fi
  \makeatletter
  \@input{\jobname.#1}%
  \if@filesw
    \@xp\newwrite\csname tf@#1\endcsname
    \immediate\@xp\openout\csname tf@#1\endcsname \jobname.#1\relax
  \fi
  \global\@nobreakfalse \endgroup
  \addvspace{32\p@\@plus14\p@}%
  \let\tableofcontents\relax
}
\def\contentsname{Contents}
\def\l@section{\@tocline{2}{.5ex}{0mm}{5pc}{}}
\def\l@subsection{\@tocline{2}{0pt}{2em}{5pc}{}}
\makeatother
\fi 

\def\to{\mathchoice{\longrightarrow}{\rightarrow}{\rightarrow}{\rightarrow}}
\makeatletter
\newcommand{\shortxra}[2][]{\ext@arrow 0359\rightarrowfill@{#1}{#2}}
\def\longrightarrowfill@{\arrowfill@\relbar\relbar\longrightarrow}
\newcommand{\longxra}[2][]{\ext@arrow 0359\longrightarrowfill@{#1}{#2}}

\def\addtagsub#1{\let\oldtf=\tagform@\def\tagform@##1{\oldtf{##1}\hbox{$_{#1}$}}}

\makeatother

\makeatletter
\def\Nopagebreak{\@nobreaktrue\nopagebreak}

\newtheoremstyle{theorem-giventitle}
        {}{}              
        {\itshape}                      
        {}                              
        {\bfseries}                     
        {.}                             
        {\thm@headsep}                             
        {\thmnote{\bfseries#3}}
\newtheoremstyle{theorem-givenlabel}
        {}{}              
        {\itshape}                      
        {}                              
        {\bfseries}                     
        {.}                             
        {\thm@headsep}                             
        {\thmname{#1}~\thmnumber{#3}\setcurrentlabel{#3}}

\newtheoremstyle{definition-giventitle}
        {}{}              
        {}                      
        {}                              
        {\bfseries}                     
        {.}                             
        {\thm@headsep}                             
        {\thmnote{\bfseries#3}}
\def\setcurrentlabel#1{\gdef\@currentlabel{#1}}

\makeatother

\newtheorem{theorem}{Theorem}[section]
\newtheorem{theoremalpha}{Theorem}

\newtheorem{lemma}[theorem]{Lemma}

\theoremstyle{definition}
\newtheorem{definition}[theorem]{Definition}

\newtheorem{example}[theorem]{Example}
\newtheorem{remark}[theorem]{Remark}

\theoremstyle{theorem-giventitle}
\newtheorem{theorem-named}{}
\theoremstyle{theorem-givenlabel}
\newtheorem{theorem-labeled}{Theorem}

\theoremstyle{definition-giventitle}
\newtheorem{definition-named}{}
\newtheorem{step-named}{}

\numberwithin{equation}{section}

\def\Z{\mathbb{Z}}

\def\R{\mathbb{R}}

\def\id{\mathrm{id}}

\def\rhot{\rho^{(2)}}

\def\Mod{\operatorname{Mod}}

\def\setminus{\smallsetminus}

\def\csimp{c^\text{simp}}
\def\cHL{c^\text{HL}}
\def\csurg{c^\text{surg}}

\def\sHL{s_\text{HL}}
\def\ssurg{s_\text{surg}}

\begin{document}

\vspace*{0mm}

\title%
[Complexities of 3-manifolds]
{Complexities of 3-manifolds from triangulations, Heegaard splittings, and surgery presentations}

\author{Jae Choon Cha}
\address{
  Department of Mathematics\\
  POSTECH\\
  Pohang 37673\\
  Republic of Korea
  \quad -- and --\linebreak
  School of Mathematics\\
  Korea Institute for Advanced Study \\
  Seoul 02455\\
  Republic of Korea
}
\email{jccha@postech.ac.kr}

\def\subjclassname{\textup{2010} Mathematics Subject Classification}
\expandafter\let\csname subjclassname@1991\endcsname=\subjclassname
\expandafter\let\csname subjclassname@2000\endcsname=\subjclassname
\subjclass{%
}


\begin{abstract}
  We study complexities of 3-manifolds defined from triangulations,
  Heegaard splittings, and surgery presentations.  We show that these
  complexities are related by linear inequalities, by presenting
  explicit geometric constructions.  We also show that our linear
  inequalities are asymptotically optimal.  Our results are used
  in~\cite{Cha:2014-1} to estimate Cheeger-Gromov $L^2$
  $\rho$-invariants in terms of geometric group theoretic and knot
  theoretic data.
\end{abstract}

\maketitle

\section{Introduction and main results}

In this paper we study the relationship between various notions of
complexities of 3-manifolds.  In what follows, we always assume that
3-manifolds are compact.

\subsubsection*{Simplicial complexity}

The first notion of complexity we consider is defined from
triangulations.  In this paper a triangulation designates a simplicial
complex structure.

\begin{definition}
  For a 3-manifold $M$, the \emph{simplicial complexity} $\csimp(M)$
  is defined to be the minimal number of 3-simplices in a
  triangulation of~$M$.
\end{definition}

A similar notion of complexity defined from more flexible
triangulations is often considered in the literature (e.g., see
\cite{Matveev-Petronio-Vesnin:2009-1, Jaco-Rubinstein-Tillman:2009-1,
  Jaco-Rubinstein-Tillman:2011-1, Jaco-Rubinstein-Tillman:2013-1}): a
\emph{pseudo-simplicial triangulation} of a $3$-manifold $M$ is
defined to be a collection of $3$-simplices together with affine
identifications of faces from which $M$ is obtained as the quotient
space.  The \emph{pseudo-simplicial complexity}, or the
\emph{complexity} $c(M)$ of $M$ is defined to be the minimal number of
$3$-simplices in a pseudo-simplicial triangulation.  For closed
irreducible 3-manifolds, $c(M)$ agrees with Matveev's
complexity~\cite{Matveev:1990-1} defined in terms of spines, unless
$M=S^3$, $\R P^3$, or~$L(3,1)$.  Since the second barycentric
subdivision of a pseudo-simplicial triangulation is a triangulation
and a 3-simplex is decomposed to $(4!)^2=576$ 3-simplices in the
second barycentric subdivision, we have
\[
\frac{1}{576}\cdot \csimp(M) \le c(M) \le \csimp(M).
\]

\subsubsection*{Heegaard-Lickorish complexity}

Recall that a Heegaard splitting of a closed 3-manifold is represented
by a mapping class in the mapping class group $\Mod(\Sigma_g)$ of a
surface $\Sigma_g$ of genus~$g$.  (Our precise convention is described
in the beginning of
Section~\ref{section:triangulation-from-heegaard-splitting}.  The
identity mapping class gives the standard Heegaard splitting of $S^3$
shown in Figure~\ref{figure:lickorish-generators}.)  It is well known
that $\Mod(\Sigma_g)$ is finitely generated; Lickorish showed that
$\Mod(\Sigma_g)$ is generated by the $\pm1$ Dehn twists about the
$3g-1$ curves $\alpha_i$, $\beta_i$, and $\gamma_i$ shown in
Figure~\ref{figure:lickorish-generators}~\cite{Lickorish:1962-1,Lickorish:1964-1}.

\begin{figure}[t]
  \labellist
  \small\hair 0mm
  \pinlabel {$\alpha_1$} at 46 70
  \pinlabel {$\beta_1$} at 34 14
  \pinlabel {$\gamma_1$} at 93 65
  \pinlabel {$\alpha_2$} at 138 70
  \pinlabel {$\beta_2$} at 124 14
  \pinlabel {$\gamma_2$} at 180 65
  \pinlabel {$\gamma_{g-1}$} at 228 65
  \pinlabel {$\alpha_g$} at 264 70
  \pinlabel {$\beta_g$} at 250 14
  \endlabellist
  \[
  \Sigma_g =\; 
  \vcenter{\hbox{\includegraphics[scale=.8]{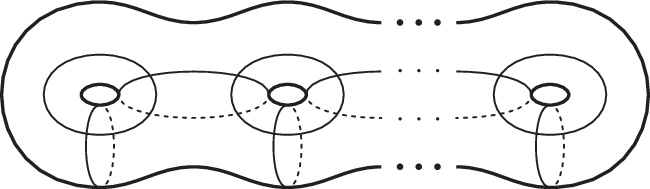}}}
  \]
  \caption{Standard Dehn twist curves of Lickorish.}
  \label{figure:lickorish-generators}
\end{figure}
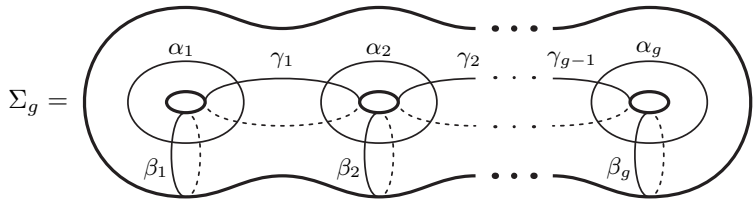

From this, a geometric group theoretic notion of complexity is defined
for 3-manifolds as follows.

\begin{definition}
  \label{definition:heegaard-lickorish-complexity}
  The \emph{Heegaard-Lickorish complexity} $\cHL(M)$ of a closed
  $3$-manifold $M$ is defined to be the minimal word length, with
  respect to the Lickorish generators, of a mapping class
  $h\in \Mod(\Sigma_g)$ on a surface $\Sigma_g$ of arbitrary genus
  which gives a Heegaard splitting of~$M$.
\end{definition}

Note that both the genus $g$ of a Heegard surface $\Sigma_g$ and the
mapping class $h$ vary in taking the minimum in
Definition~\ref{definition:heegaard-lickorish-complexity}.  By
definition, $\cHL(S^3)=0$.

We remark that the Heegaard-Lickorish complexity tells us more
delicate information than the Heegaard genus.  It turns out that the
difference of the Heegaard-Lickorish complexities of two $3$-manifolds
with the same Heegaard genus can be arbitrarily large, whereas the
Heegaard genus of a 3-manifold is bounded by twice its
Heegaard-Lickorish complexity.  See
Lemma~\ref{lemma:heegaard-genus-bound-from-Heegaard-Lickorish-complexity}
and related discussions in
Section~\ref{section:triangulation-from-heegaard-splitting}.

Our first result is the following relationship between the two
complexities defined above.

\begin{theoremalpha}
  \label{theorem:simplicial-and-HL-complexity}
  For any closed 3-manifold $M\ne S^3$, $\csimp(M) \le 552\cdot
  \cHL(M)$.
\end{theoremalpha}

We remark that upper bounds to (pseudo-)simplicial complexity in terms
of a Heegaard splitting were studied earlier in the literature, for
instance see \cite[Proposition~3]{Matveev:1990-1} and
\cite[Proposition~2.1.8]{Matveev:2007-1}.  In many cases
Theorem~\ref{theorem:simplicial-and-HL-complexity} provides a sharper
upper bound.  For more about this, see
Remark~\ref{remark:other-bounds-from-heegaard} as well as
Theorems~\ref{theorem:asymptotic-growth}
and~\ref{theorem:linear-bounds-are-optimal} below which concern the
optimality of our bound.  The optimality is essential in an
application of~\cite{Cha:2014-1} (see the last part of the
introduction).

\subsubsection*{Surgery complexity}

To define another notion of complexity of 3-manifolds from knot
theoretic information, we consider Dehn surgery with integral
coefficients.  For a framed link $L$ in $S^3$, let
$f(L)=\sum_i |f_i(L)|$ where $f_i(L)\in\Z$ is the framing on the $i$th
component of~$L$.  If a component $K$ of $L$ is contained in an
embedded 3-ball in $S^3$ which is disjoint from other components, then
we call $K$ a \emph{split component}.  Let $n(L)$ be the number of
split unknotted zero framed components of~$L$.  An example with
$f(L)=2$, $n(L)=1$ is illustrated in
Figure~\ref{figure:framed-link-example}.  We denote by $c(L)$ the
\emph{crossing number} of a link $L$ in $S^3$, that is, $c(L)$ is the
minimal number of crossings in a planar diagram of~$L$.  As a
convention, if $L$ is empty, then $c(L)=f(L)=n(L)=0$.

\begin{figure}[H]
  \centering
  \begin{tikzpicture}[
    x=1.1mm,y=1.1mm, line width=1.2pt,
    over/.style={draw=white,double=black,double distance=1.2pt,line width=2pt},
    thinover/.style={draw=white,double=black,double distance=.4pt,line width=1.1pt},
    hidden/.style={text=white,opacity=0}
    ]
    \small
    \def\c[#1]#2{coordinate(#2) node[hidden,#1]{\footnotesize$#2$}}
    \draw[over] (-13,10) arc(180:0:4 and 2) \c[]{G};
    \draw[over] (5,5) arc(180:0:9 and 2) \c[]{H};
    \draw[over] (5,0) \c[]{C} -- (0,0) arc(270:90:9 and 10) \c[]{A} 
    ++(8,-5) \c[]{F} ..controls+(270:6)and+(90:10).. ++(-8,-14.5) \c[]{B};
    \draw[over] (B) arc(180:360:5) ..controls+(90:10)and+(270:6).. ++(-8,14.5) \c[]{D}
    ..controls+(90:3)and+(-5,0).. ++(8,5) \c[]{E};
    \draw[over] (E) arc(90:-90:9 and 10) -- (C);
    \draw[over] (A) node[anchor=south]{$-3$} ..controls+(5,0)and+(90:3).. (F);
    \draw[over] (G) arc(0:-180:4 and 2) node[anchor=east]{$1$};
    \draw[over] (H) node[anchor=west]{$0$} arc(0:-180:9 and 2);
    \draw[over] (27,15) circle(4) ++(4,0) node[anchor=west]{$0$};
  \end{tikzpicture}
  \caption{A framed link $L$ with $f(L)=2$, $n(L)=1$.}
  \label{figure:framed-link-example}
\end{figure}
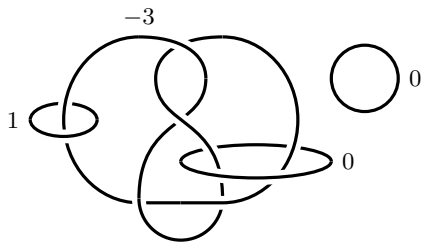

\begin{definition}
  The \emph{surgery complexity} of a closed 3-manifold $M$ is defined by
  \[
  \csurg(M)= \min_L \{2c(L)+f(L)+n(L)\}
  \]
  where $L$ varies over framed links in $S^3$ from which $M$ is
  obtained by surgery.
\end{definition}

We remark that we bring in $n(L)$ to detect $S^1\times S^2$ summands,
which can be added to any 3-manifold by connected sum without altering
$c(L)$ and $f(L)$ of a framed link $L$ giving the 3-manifold.  Note
that $n(L)=0$ for any $L$ that gives $M$ if $M$ has no $S^1\times S^2$
summand.  In particular it is the case if $M$ is irreducible.  Note
that $\csurg(S^3)=0$ by our convention.

Our second result is the following relationship between the simplicial
complexity and the surgery complexity.

\begin{theoremalpha}
  \label{theorem:simplicial-and-surgery-complexity}
  For any closed 3-manifold $M\ne S^3$, $\csimp(M) \le 72\cdot
  \csurg(M)$.
\end{theoremalpha}

We remark that Matveev gave a similar inequality which relates the
complexity $c(M)$ to a surgery presentation
\cite[Proposition~2.1.13]{Matveev:2007-1}

The proofs of Theorems~\ref{theorem:simplicial-and-HL-complexity}
and~\ref{theorem:simplicial-and-surgery-complexity} consist of
geometric arguments which explicitly construct triangulations from
Heegaard splittings and from surgery presentations.  Details are given
in Sections~\ref{section:triangulation-from-surgery-presentation}
and~\ref{section:triangulation-from-heegaard-splitting}.

\subsubsection*{Optimality of Theorems~\ref{theorem:simplicial-and-HL-complexity}
and~\ref{theorem:simplicial-and-surgery-complexity}}

It is natural to ask how sharp the inequalities in
Theorems~\ref{theorem:simplicial-and-HL-complexity}
and~\ref{theorem:simplicial-and-surgery-complexity} are.  This seems
to be a nontrivial problem, since it appears to be hard to determine
the complexities we consider, or even to find an efficient lower bound
for them.  We remark that the determination and lower bound problems
for the pseudo-simplicial complexity $c(M)$ have been studied
extensively in the literature and regarded as difficult
problems~\cite{Matveev:2003-1,Jaco-Rubinstein-Tillman:2013-1}.

We show that the \emph{linear} inequalities in
Theorems~\ref{theorem:simplicial-and-HL-complexity}
and~\ref{theorem:simplicial-and-surgery-complexity} are
\emph{asymptotically} optimal.  This can be described in terms of
standard notations for asymptotic growth, as follows.  Recall that we
write $f(n)\in O(g(n))$ if $f$ is \emph{bounded above} by $g$
asymptotically, that is, $\limsup_{n\to\infty} |f(n)/g(n)|$ is finite.
Also, $f(n)\in o(g(n))$ if $f(n)$ is \emph{dominated by} $g(n)$
asymptotically, that is, $\limsup_{n\to\infty} |f(n)/g(n)|=0$.  We
write $f(n)\in \Omega(g(n))$ if $f(n)$ is not dominated by~$g(n)$.

Define two functions $\sHL(\ell)$ and $\ssurg(k)$ by
\begin{align*}
  \sHL(\ell) &= \sup\{\csimp(M) \mid \cHL(M)\le \ell\},
  \\
  \ssurg(k) &= \sup\{\csimp(M) \mid \csurg(M)\le k\},
\end{align*}
where the supremums exist by
Theorems~\ref{theorem:simplicial-and-HL-complexity}
and~\ref{theorem:simplicial-and-surgery-complexity}\@.  In other
words, $\sHL(\ell)$ is the ``largest possible value'' of the
simplicial complexity for 3-manifolds with Heegaard-Lickorish
complexity $\ell$ or less.  We can interpret $\ssurg(k)$ similarly.

\begin{theoremalpha}
  \label{theorem:asymptotic-growth}
  $\sHL(\ell) \in O(\ell)\cap \Omega(\ell)$ and $\ssurg(k) \in O(k)
  \cap \Omega(k)$.
\end{theoremalpha}

As explicit examples, the
lens spaces $L(n,1)$ satisfy the following:

\begin{theoremalpha}
  \label{theorem:linear-bounds-are-optimal}
  For any $n>3$,
  \begin{gather*}
    \frac{1}{4357080} \cdot \cHL(L(n,1)) \le \csimp(L(n,1)),
    \\
    \frac{1}{4357080} \cdot \csurg(L(n,1)) \le \csimp(L(n,1)).
  \end{gather*}
\end{theoremalpha}

We also prove a similar inequality for a larger class of 3-manifolds.
See Theorem~\ref{theorem:surgery-manifold-inequality} and related
discussions in Section~\ref{section:linear-bounds-are-optimal}.

The proofs of Theorems~\ref{theorem:asymptotic-growth}
and~\ref{theorem:linear-bounds-are-optimal} are given in
Section~\ref{section:linear-bounds-are-optimal}.

\subsubsection*{Applications to universal bounds for Cheeger-Gromov
  invariants}

Results in this paper are closely related to the recent development of
a topological approach to the universal bounds of Cheeger-Gromov $L^2$
$\rho$-invariants in~\cite{Cha:2014-1}.  In fact,
Theorems~\ref{theorem:simplicial-and-HL-complexity}
and~\ref{theorem:simplicial-and-surgery-complexity} of this paper are
used as essential ingredients in~\cite{Cha:2014-1} to give explicit
linear estimates of Cheeger-Gromov $\rho$-invariants of 3-manifolds in
terms of geometric group theoretical and knot theoretical data.  See
Theorems~1.8 and~1.9 of~\cite{Cha:2014-1}.  This application is a
major motivation of the present paper.  Our inequalities in
Theorems~\ref{theorem:simplicial-and-HL-complexity}
and~\ref{theorem:simplicial-and-surgery-complexity} are sharp enough,
compared with earlier similar work, to give results that the linear
estimates in~\cite{Cha:2014-1} are \emph{asymptotically optimal}.  See
Theorem~7.8 of~\cite{Cha:2014-1}.

On the other hand, the lower bounds in
Theorem~\ref{theorem:linear-bounds-are-optimal} are proven by
employing results of~\cite{Cha:2014-1} which relate triangulations and
the Cheeger-Gromov $\rho$-invariants.  See
Section~\ref{section:linear-bounds-are-optimal} for more details.

\subsection*{Acknowledgements}

The author thanks an anonymous referee for comments which were very
helpful in improving results and in fixing a mistake of an earlier
version of this paper.  This work was partially supported by NRF
grants 2013067043 and 2013053914.

\section{Linear complexity triangulations from surgery presentations}
\label{section:triangulation-from-surgery-presentation}

In this section we present a construction of a triangulation from a
surgery presentation.


\def\kinkp{\lower.8pt\hbox{\tikz[x=.5ex,y=.5ex,line width=.8pt]{
    \draw (0,0) ..controls+(2,0) and +(1.8,0).. (1.5,3);
    \draw[draw=white,double=black,double distance=.8pt,line width=.8pt]
    (1.5,3) ..controls+(-1.8,0) and +(-2,0).. (3,0);
  }}}
\def\kinkn{\lower.8pt\hbox{\tikz[x=.5ex,y=.5ex,xscale=-1,line width=.8pt]{
    \draw (0,0) ..controls+(2,0) and +(1.8,0).. (1.5,3);
    \draw[draw=white,double=black,double distance=.8pt,line width=.8pt]
    (1.5,3) ..controls+(-1.8,0) and +(-2,0).. (3,0);
  }}}
  
\begin{lemma}
  \label{lemma:triangulation-for-blackboard-framing-surgery}
  Suppose $L$ is a framed link in~$S^3$.  Suppose there is a planar
  diagram $D$ with $c$ or fewer crossings for $L$, in which there is
  no local kink \textup{(}\kinkp, \kinkn\textup{)} and each zero
  framed component of $L$ is involved in a crossing.  Let $w_i\in \Z$
  be the writhe of the $i$th component in the diagram~$D$.  Then the
  3-manifold $M$ obtained by surgery on $L$ has simplicial complexity
  at most~$96c+48\sum|f_i(L)-w_i|$.
\end{lemma}



\begin{example}
  \label{example:universal-bound-for-6_1}
  Consider the stevedore knot, which is $6_1$ in the table in
  Rolfsen~\cite{Rolfsen:1976-1}, or
  KnotInfo~\cite{cha-livingston:knotinfo}.  It has a planar diagram
  with 6 crossings, where 2 of them have the same sign but the other 4
  have the opposite sign.  It follows
  that the zero surgery manifold $M$ of $6_1$ satisfies $\csimp(M) \le
  96\cdot 6 + 48\cdot 2 = 672$.
\end{example}

Before we prove
Lemma~\ref{lemma:triangulation-for-blackboard-framing-surgery}, we
prove Theorem~\ref{theorem:simplicial-and-surgery-complexity} using
Lemma~\ref{lemma:triangulation-for-blackboard-framing-surgery}.

\begin{proof}
  [Proof of Theorem~\ref{theorem:simplicial-and-surgery-complexity}]
  Recall that Theorem~\ref{theorem:simplicial-and-surgery-complexity} says
  \[
  \csimp(M) \le 72\cdot \csurg(M)
  \]
  for $M\ne S^3$.

  We need the following two observations: firstly, we have
  \begin{equation}
    \label{equation:subadditivity-complexity}
    \csimp(M_1\# M_2) \le \csimp(M_1)+\csimp(M_2)-2,
  \end{equation}
  since the connected sum of two triangulated 3-manifolds can be
  performed by deleting a 3-simplex from each and then glueing faces.
  Second, we have
  \begin{equation}
    \label{equation:S^1-times-S^2-estimate}
    \csimp(S^1\times S^2)\le 72.
  \end{equation}
  For instance, by taking the product of a triangle triangulation of
  $S^1$ and its suspension which is a triangulation of $S^2$, and then
  by applying the standard prism decomposition to each product
  $\Delta^1\times \Delta^2$ (see
  Figure~\ref{figure:prism-decomposition}), we obtain a triangulation
  of $S^1\times S^2$ with $3\cdot 6\cdot 3 = 54$ tetrahedra.

  \begin{figure}[H]
    \includegraphics[scale=.8]{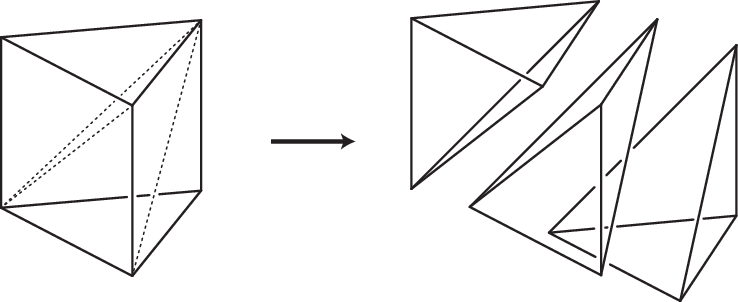}
    \caption{Prism decomposition of $\Delta^1\times\Delta^2$.}
    \label{figure:prism-decomposition}
  \end{figure}

  Choose a framed link $L$ such that $M$ is obtained by surgery on~$L$
  and $2c(L)+f(L)+n(L)=\csurg(M)$.  Choose a planar diagram $D$ for
  $L$ with minimal number of crossings, and let $D_0$ be the
  subdiagram of $D$ consisting of zero framed split components of~$L$.
  From the minimality, it follows that $D$ has no local kink and $D_0$
  consists of zero framed circles with no crossing.  Also every zero
  framed component of $L$ in $D-D_0$ is involved in a crossing.  Let
  $M'$ be the 3-manifold obtained by surgery along the given framing
  of~$D-D_0$.  Since a component of $D_0$ contributes an
  $S^1\times S^2$ summand,
  $M=M'\mathbin{\#} (n(L)\cdot(S^1\times S^2))$.

  If $D=D_0$, then $c(L)=f(L)=0$ and $M'=S^3$; also, $n(L)\ge 1$
  since $M\ne S^3$.  It follows that
  \[
  \csimp(M)\le n(L) \cdot \csimp(S^1\times S^2) \le 72\cdot n(L)
  \]
  by using \eqref{equation:subadditivity-complexity} and
  \eqref{equation:S^1-times-S^2-estimate}.  This is the desired
  conclusion for this case.

  If $D\ne D_0$, we have
  \begin{equation}
    \label{equation:removing-exceptional-components}
    \csimp(M) \le \csimp(M') + n(L)\cdot \csimp(S^1\times S^2) \le
    \csimp(M')+72\cdot n(L)
  \end{equation}
  by using \eqref{equation:subadditivity-complexity} and
  \eqref{equation:S^1-times-S^2-estimate}.  The number of crossings of
  $D-D_0$ is equal to that of $D$, which is equal to $c(L)$ by our
  choice of $D$.  Let $w_i$ be the writhe of the $i$th component in
  $D-D_0$, and $f_i$ be its given framing.  Since a crossing in the
  diagram contributes $1$, $0$, or $-1$ to $w_i$ for some $i$, it
  follows that $\sum |w_i| \le c(L)$.  Therefore we have
  \begin{equation}
    \label{equation:complexity-nonexceptional-part}
    \begin{aligned}
      \csimp(M') &\le 96\cdot c(L) + 48\cdot \sum |f_i-w_i| \\
      &\le  96\cdot c(L) + 48\cdot(f(L)+c(L)) \le 72\cdot (2c(L)+f(L))     
    \end{aligned}
  \end{equation}
  by Lemma~\ref{lemma:triangulation-for-blackboard-framing-surgery}.
  From \eqref{equation:removing-exceptional-components} and
  \eqref{equation:complexity-nonexceptional-part}, the desired
  conclusion follows.
\end{proof}

\begin{proof}[Proof of
  Lemma~\ref{lemma:triangulation-for-blackboard-framing-surgery}]

  We will construct a triangulation of the exterior of $L$ which is
  motivated from J. Weeks' SnapPea (see~\cite{Weeks:2005-1}), and then
  will triangulate the Dehn filling tori in a compatible way.

  In what follows we view $D$ as a planar diagram lying on~$S^2$.  By
  the subadditivity~\eqref{equation:subadditivity-complexity}, we may
  assume that the diagram $D$ is nonsplit, that is, any simple closed
  curve in $S^2$ disjoint from $D$ bounds a disk disjoint from~$D$.
  This is equivalent to that every region of $D$ is a disk.


  Either $D$ has at least one crossing, or $D$ is a circle with no
  crossings.  First, suppose that it is the former case.

  Consider the dual graph $G_0$ of $D$, whose regions are quadrangles
  corresponding to crossings.  (Since $D$ has no local kinks, the four
  vertices of each quadrangles are mutually distinct.)  For each
  component of the link $L$, choose an edge of $G_0$ which is dual to
  the component (that is, the edge intersects a strand of $D$ that
  belongs to the component), and add $2|f_i-w_i|$ parallels of the
  edge, where $f_i=f_i(L)$ is the framing and $w_i\in \Z$ is the
  writhe of the component.  Denote the resulting graph by~$G$.  For an
  example, see the left of Figure~\ref{figure:exterior-decomposition},
  which illustrates the case of a $(+1)$-framed figure eight.

  View the link $L$ as a submanifold of $S^2\times[-1,1]$ which
  projects to $D$ under $S^2\times[-1,1] \to S^2$, and remove from
  $S^2\times[-1,1]$ an open tubular neighborhood $\nu(L)$ of $L$ which
  is tangential to $S^2\times\{-1,1\}$ at
  $(\text{each crossing})\times\{-1,1\}$; cutting along
  $G\times [-1,1]$, we obtain pieces of two types: (i) cubes with two
  tunnels, which correspond to the crossings of $D$, and (ii) those of
  the form $(\text{2-gon})\times[-1,1]$ with a tunnel removed, which
  correspond to the edges of~$G-G_0$.  See the middle of
  Figure~\ref{figure:exterior-decomposition}.

  \begin{figure}[ht]
    \labellist\small
    \pinlabel{$D$} at 85 235
    \pinlabel{$G$} at 20 136
    \pinlabel{type (i)} at 240 140
    \pinlabel{type (ii)} at 240 20
    \endlabellist
    \includegraphics{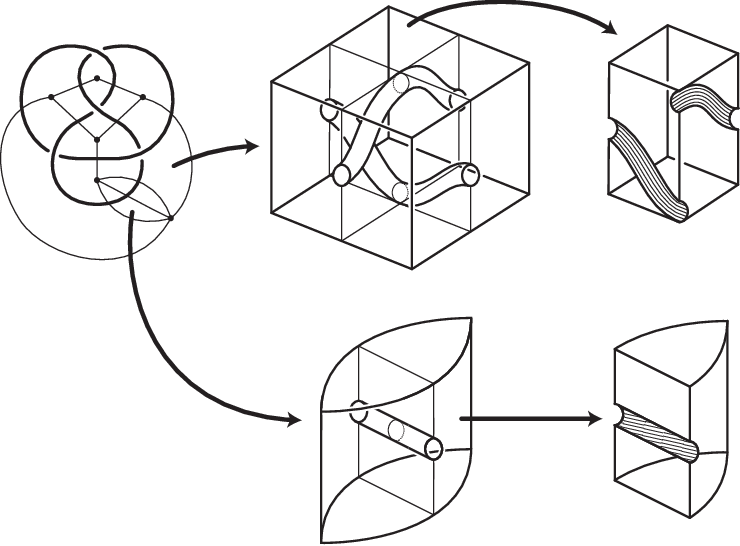}
    \caption{A decomposition of a link diagram.}
    \label{figure:exterior-decomposition}
  \end{figure}
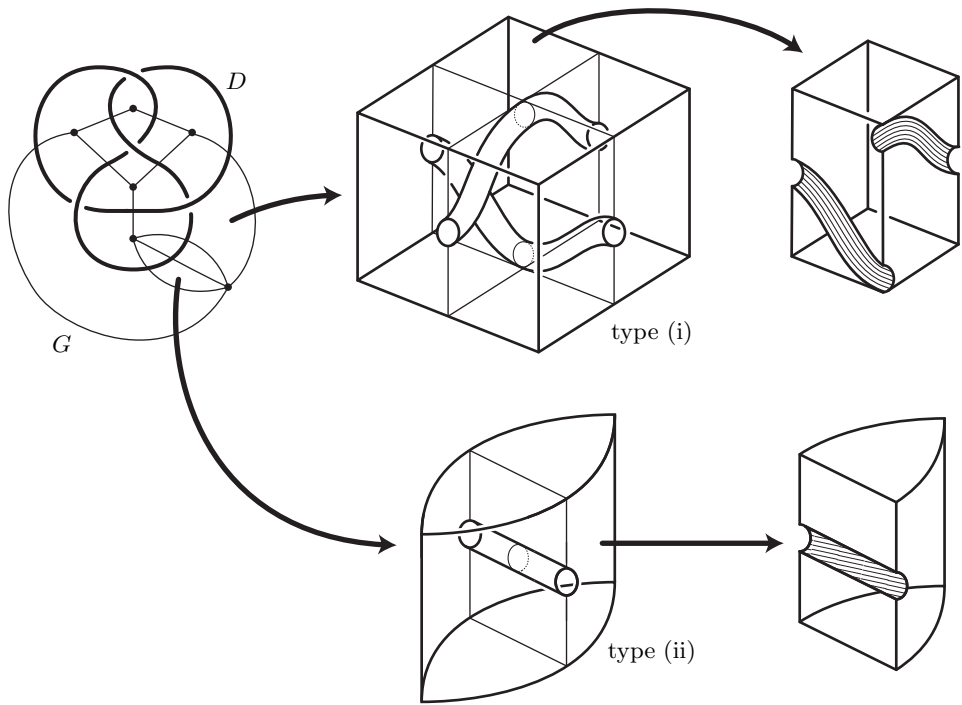

  Cut each piece along $D\times[-1,1]$.  In case of type (i), we
  obtain 4 equivalent subpieces.  See the top right of
  Figure~\ref{figure:exterior-decomposition}.  The hatched quadrangles
  represent~$\partial \nu(L)$.  Each of the 4 subpieces can be viewed
  as a cube shown in the left of
  Figure~\ref{figure:subpiece-decomposition}.  Let $p$ be the vertex
  shown in Figure~\ref{figure:subpiece-decomposition}, and triangulate
  the three square faces not adjacent to~$p$ as in the left of
  Figure~\ref{figure:subpiece-decomposition}.  By taking a cone from
  $p$, we obtain a triangulation of the each type (i) subpiece.  Since
  the triangulation of the faces away from $p$ has 14 triangles, the
  cone triangulation of a type (i) subpiece has 14 tetrahedra.

  \begin{figure}[t]
    \labellist
    \footnotesize\pinlabel{$p$} at 49 55
    \footnotesize\pinlabel{$q$} at 237 68
    \endlabellist    
    \includegraphics[scale=.9]{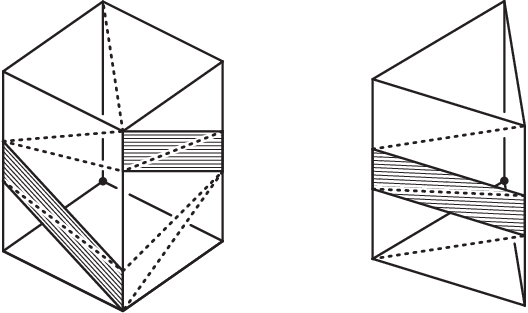}
    \caption{Decomposition of subpieces.}
    \label{figure:subpiece-decomposition}
  \end{figure}

  In case of type (ii), by cutting each piece along $D\times[-1,1]$,
  we obtain two subpieces each of which are as in the bottom right of
  Figure~\ref{figure:exterior-decomposition}.  For each type (ii)
  subpiece, triangulate the front face as shown in the right of
  Figure~\ref{figure:subpiece-decomposition}, and then triangulate the
  subpiece by taking the cone of the union of the front face and top
  triangle from the vertex $q$, similarly to the above type (i) case.
  This triangulation of a type (ii) subpiece has 7 tetrahedra.

  Suppose $D$ has $c$ crossings.  For brevity, denote
  $\delta := \sum |f_i-w_i|$.  There are $4c$ subpieces of type (i)
  and $4\delta$ subpieces of type~(ii).  By applying the above to each
  of them, we obtain a triangulation of
  $S^2\times[-1,1]\setminus \nu(L)$, which has
  $14\cdot 4c + 7\cdot 4\delta = 56c+28\delta$ tetrahedra.

  For $t=\pm1$, the triangulation restricts to a triangulation of
  $S^2\times \{t\}$ with $8c+4\delta$ triangles, since the top of type
  (i) and (ii) subpieces consist of two triangles and a single
  triangle respectively.  Attaching two $3$-balls triangulated as the
  cone of these triangulations, we obtain a triangulation of
  $S^3\setminus \nu(L)$ which has
  $(56c+28\delta)+ 2\cdot(8c+4\delta) = 72c+36\delta$ tetrahedra.

  In our triangulation, there are $8c+4\delta$ hatched quadrangular
  regions, and they are paired up to form $4c+2\delta$ annuli, and the
  $i$th boundary component of $\nu(L)$ is a union of $2k_i+2|f_i-w_i|$
  such annuli, where $k_i$ is the number of times the $i$th component
  of $L$ passes through a crossing.  We have $\sum k_i = 2c$.  (Since
  a component may pass through the same crossing twice, $k_i$ may not
  be equal to the number of crossings that the component passes
  through.)  See the left of
  Figure~\ref{figure:boundary-triangulation}; the hatched meridional
  annulus is one of these $2k_i+2|f_i-w_i|$ annuli.

  \begin{figure}[H]
    \labellist
    \small
    \pinlabel{$\alpha_i$} at 33 107
    \pinlabel{$\alpha'_i$} at 17 74
    \pinlabel{\footnotesize
      \begin{tabular}{c}from\\type (i)\\(untwisted)\end{tabular}} at 60 50
    \pinlabel{\footnotesize
      \begin{tabular}{c}from\\type (ii)\\(twisted)\end{tabular}} at 105 50
    \pinlabel{\footnotesize
      \begin{tabular}{c}from\\type (i)\\(untwisted)\end{tabular}} at 150 50
    \pinlabel{
      \begin{tabular}{c}
        the $i$th\\component\\of $\partial\nu(L)$
      \end{tabular}} at 193 68
    \pinlabel{$\alpha_i$} at 296 74
    \pinlabel{$\alpha'_i$} at 295 156
    \pinlabel{$D^2\times S^1$} at 255 66
    \pinlabel{$\alpha_i$} at 267 8
    \pinlabel{$\alpha'_i$} at 365 8
    \endlabellist
    \includegraphics{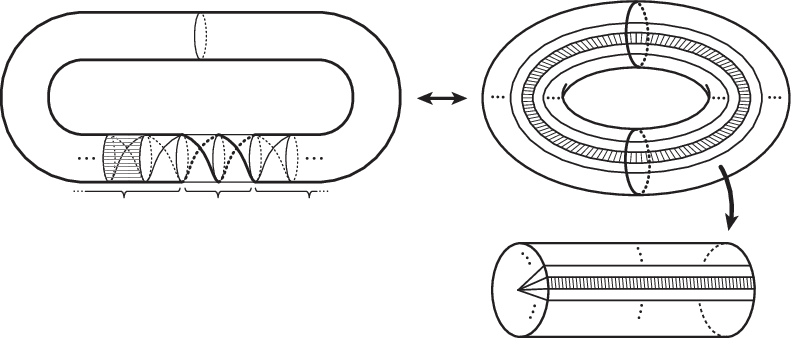}
    \caption{A boundary component and a Dehn filling torus.}
    \label{figure:boundary-triangulation}
  \end{figure}
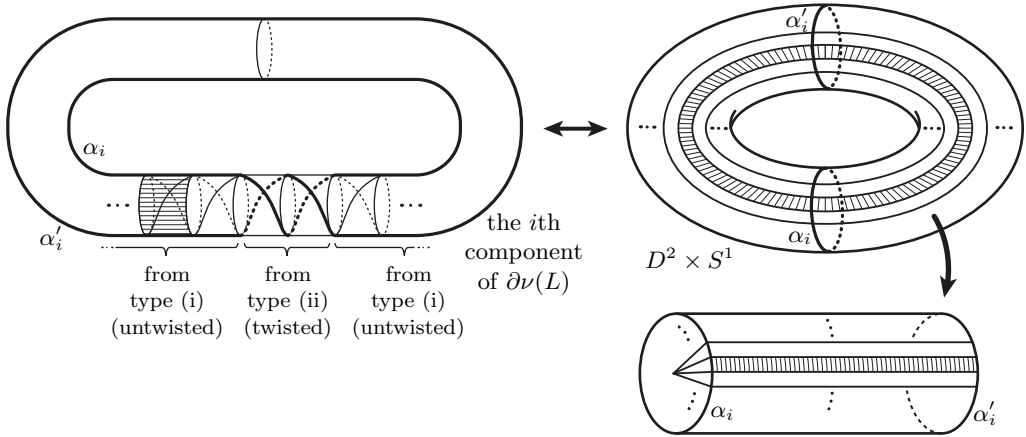

  On the $i$th boundary component of~$\nu(L)$, take the top and bottom
  edges of the hatched quadrangles in type (i) subpieces, and the
  diagonal edges used to triangulate the hatched quadrangles in type
  (ii) subpieces.  We may assume that the union of these edges
  consists of two parallel circles, say $\alpha_i$ and $\alpha'_i$, by
  appropriately altering the choices of diagonals used above to
  triangulate the hatched quadrangles if necessary.  See the left of
  Figure~\ref{figure:boundary-triangulation} in which $\alpha_i$ and
  $\alpha'_i$ are shown as thick curves.  Moreover we may assume that
  the framing represented by $\alpha_i$ differs from the blackboard
  framing by $f_i-w_i$; that is, whenever $\alpha_i$ passes through a
  type (ii) piece, a half twist with the same sign as that of
  $f_i-w_i$ is introduced with respect to the blackboard framing,
  while $\alpha_i$ runs along the blackboard framing in type (i)
  pieces.  See the left of Figure~\ref{figure:boundary-triangulation},
  which illustrates the case of $f_i-w_i=1$.  Since the blackboard
  framing is equal to $w_i$, it follows that $\alpha_i$ represents the
  given framing~$f_i$.


  Take a solid torus $D^2\times S^1$ for each component of~$L$.
  Attach the solid tori to the exterior $S^3\setminus \nu(L)$ along
  orientation reversing homeomorphisms of boundary tori which takes
  the curves $\alpha_i$ and $\alpha'_i$ to meridians bounding disks
  and takes a hatched annulus to a longitudinal annulus, as shown in
  Figure~\ref{figure:boundary-triangulation}.  Pulling back the
  triangulation of $\partial(S^3\setminus\nu(L))$, we obtain a
  triangulation of $\partial(D^2\times S^1$).  It extends to a
  triangulation of $D^2\times S^1$ as follows.  By cutting the
  $D^2\times S^1$ along the meridional disks bounded by $\alpha_i$ and
  $\alpha'_i$, we obtain two solid cylinders $D^2\times [0,1]$.  Note
  that we already have $2k_i+2|f_i-w_i|$ vertices on
  $\partial D^2\times 0$.  We triangulate $D^2\times 0$ into
  $2k_i+2|f_i-w_i|$ triangles, by drawing edges joining the vertices
  to the center of $D^2\times 0$.  See the bottom of
  Figure~\ref{figure:boundary-triangulation}.  Taking the product with
  $[0,1]$, we decompose $D^2\times[0,1]$ into $2k_i+2|f_i-w_i|$
  triangular prisms.  Note that each prism corresponds to a hatched
  quadrangle.  Finally we apply the standard prism decomposition
  (Figure~\ref{figure:prism-decomposition}) to each prism.  Since each
  prism gives 3 tetrahedra and there are $8c+4\delta$ hatched
  quadrangles, the union of all the Dehn filling solid tori is
  decomposed into $3(8c+4\delta)=24c+12\delta$ tetrahedra.

  The triangulation of our surgery manifold $M$ is obtained by
  adjoining the Dehn filling tori triangulations to that of the
  exterior.  By the above tetrahedra counting, it follows that the
  number of tetrahedra in $M$ is at most
  $(72c+36\delta)+(24c+12\delta) = 96c+48\delta$.  This completes the
  proof when there is at least one crossing in~$D$.

  Now, suppose $D$ consists of a single circle without crossings.
  Note that the writhe is zero in this case.  Let $f_1\in \Z$ be the
  given framing.  By the hypothesis, $f_1\ne 0$.  We need to prove
  that $\csimp(M)\le 48|f_1|$.  If $f_1=\pm1$, then $M=L(f_1,1)=S^3$,
  and it is straightforward to verify that $\csimp(S^3)\le 48$.  (For
  instance, triangulate the equator $S^2\subset S^3$ into 4 triangles,
  by viewing it as the boundary of a 3-simplex, and triangulate the
  upper and lower hemispheres by taking a cone of the equator, to
  obtain a triangulation of $S^3$ with $8$ tetrahedra.)  Suppose
  $|f_1|\ge 2$.  Note that the dual graph $G_0$ of $D$ consists of two
  vertices and a single edge joining them.  Let $G$ be the graph
  obtained by adding $2|f_1|-1$ parallels of the edge, that is, $G$
  consists of $2|f_1|$ edges between the two vertices.  Apply the same
  construction as above, using this $G$, to triangulate~$M$.  In this
  case we have $2|f_1|$ type (ii) pieces and no type (i) pieces.
  Using $|f_1|\ge 2$, it is verified that our construction produces a
  simplicial complex structure.  (No two vertices of a tetrahedron are
  identified and each tetrahedron is uniquely determined by its
  vertices.)  By the above counting, the number of tetrahedra
  is~$48|f_1|$, as desired.
\end{proof}

\section{Linear complexity triangulations from Heegaard splittings}
\label{section:triangulation-from-heegaard-splitting}

In this section we present an explicit construction of a triangulation
from a Heegaard splitting given by a mapping class. Recall from
Definition~\ref{definition:heegaard-lickorish-complexity} that the
\emph{Heegaard-Lickorish complexity} of a closed $3$-manifold $M$ is
the minimal word length, in the Lickorish generators, of a mapping
class on an arbitrary surface which gives a Heegaard splitting of~$M$.
Here the Lickorish generators of the mapping class group
$\Mod(\Sigma_g)$ of an oriented surface $\Sigma_g$ of genus $g$ are
defined to be the $\pm1$ Dehn twists along the curves
$\alpha_1,\ldots,\alpha_g$, $\beta_1,\ldots,\beta_g$,
$\gamma_1,\ldots,\gamma_{g-1}$ shown in
Figure~\ref{figure:lickorish-generators}.

To make it precise, we use the following convention.  Fix a standard
embedding of a surface $\Sigma_g$ of genus $g$ in $S^3$ as in
Figure~\ref{figure:lickorish-generators}.  Then $\Sigma_g$ bounds the
inner handlebody $H_1$ and the outer handlebody $H_2$ in~$S^3$.  Let
$i_j\colon \Sigma_g\to H_j$ ($j=1,2$) be the inclusion.  The mapping
class $h\in\Mod(\Sigma_g)$ of a homeomorphism
$f\colon \Sigma_g\to \Sigma_g$ gives a Heegaard splitting
$(\Sigma_g, \{\beta_i\}, \{f(\alpha_i)\})$ of the $3$-manifold
\[
M = (H_1 \cup H_2)/i_1(f(x))\sim i_2(x),\ x\in \Sigma_g.
\]
In other words, $M$ is obtained by attaching $g$ $2$-handles to the
inner handlebody $H_1$ with boundary $\Sigma_g$ along the curves
$f(\alpha_i)$ and then attaching a $3$-handle.  Under our convention,
the identity mapping class gives us~$S^3$.

The Heegaard-Lickorish complexity can be compared with the Heegaard
genus by the following lemma.

\begin{lemma}
  \label{lemma:heegaard-genus-bound-from-Heegaard-Lickorish-complexity}
  Suppose $M$ is a closed $3$-manifold with a Heegaard splitting given
  by a mapping class $h\in \Mod(\Sigma_g)$ which is a product of
  $\ell$ Lickorish generators.  Then for some $g'\le 2\ell$, $M$
  admits a Heegaard splitting given by a mapping class
  $h' \in \Mod(\Sigma_{g'})$ which is a product of $\ell$ Lickorish
  generators.
\end{lemma}

From
Lemma~\ref{lemma:heegaard-genus-bound-from-Heegaard-Lickorish-complexity},
it follows immediately that the Heegaard genus is not greater than
twice the Heegaard-Lickorish complexity.  On the other hand, it is
easily seen that a 3-manifold may be drastically more complicated than
another with the same Heegaard genus.  For example, all the lens
spaces $L(n,1)$ have Heegaard genus one, but $L(n,1)$ is represented
by a genus one mapping class of Heegaard-Lickorish word length~$n$.
In fact, by results of \cite{Cha:2014-1} (see also
Lemma~\ref{lemma:lens-space-complexity} and related discussions in the
present paper), $\csimp(L(n,1)) \to \infty$ as $n\to\infty$, and
consequently $\cHL(L(n,1)) \to \infty$ and $\csurg(L(n,1)) \to \infty$
by Theorems~\ref{theorem:simplicial-and-HL-complexity}
and~\ref{theorem:simplicial-and-surgery-complexity}\@.

\begin{proof}[Proof of
  Lemma~\ref{lemma:heegaard-genus-bound-from-Heegaard-Lickorish-complexity}]

  For a Lickorish generator $t\in \Mod(\Sigma_g)$, we say that $t$
  \emph{passes through the $i$th hole of $\Sigma_g$} if $t$ is a Dehn
  twist along either one of the curves $\alpha_i$, $\beta_i$,
  $\gamma_i$ or $\gamma_{i-1}$ (see
  Figure~\ref{figure:lickorish-generators}).  It is easily seen from
  Figure~\ref{figure:lickorish-generators} that a Lickorish generator
  can pass through at most two holes of~$\Sigma_g$.  Therefore, the
  Lickorish generators which appear in the given word expression of
  $h$ of length $\ell$ can pass through at most $2\ell$ holes.  If
  $g>2\ell$, then for some $i$, no Lickorish generator used in $h$
  passes through the $i$th hole.  By a destabilization which removes
  the $i$th hole from $\Sigma_g$, we obtain a Heegaard splitting of
  $M$ of genus $g-1$ given by a mapping class which is a product of
  $\ell$ Lickorish generators.  By an induction, the proof is
  completed.
\end{proof}

Lickorish's work \cite{Lickorish:1962-1,Lickorish:1964-1} presents a
construction of a surgery presentation from a Heegaard splitting.
From his proof, we obtain the following:

\begin{theorem}
  \label{theorem:surgery-and-HL-complexity}
  For any closed 3-manifold $M$, $\csurg(M) \le 2\cdot\cHL(M)^2 +
  3\cdot\cHL(M)$.
\end{theorem}

\begin{proof}
  Suppose $M$ has a Heegaard splitting represented by a mapping class
  of Lickorish word length~$\ell$.  By the arguments in
  Lickorish~\cite{Lickorish:1962-1,Lickorish:1964-1} (see also
  Rolfsen's book \cite[Chapter~9, Section I]{Rolfsen:1976-1}), $M$ is
  obtained by surgery on a link $L$ with $\ell$ $(\pm1)$-framed
  components, which admits a planar diagram in which no component has
  a self-crossing and any two distinct components have at most two
  crossings between them.  See
  Figure~\ref{figure:surgery-link-example} for an example.  It follows
  that $n(L)=0$, $f(L)=\ell$, and
  $c(L)\le 2\cdot \binom{\ell}{2} = \ell(\ell+1)$.  By definition, we
  have $\csurg(M) \le 2c(L)+f(L)+n(L) \le 2\ell^2+3\ell$.
\end{proof}

\begin{figure}[ht]
  \includegraphics{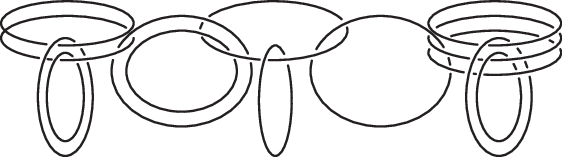}
  \caption{An example of Lickorish's surgery link.}
  \label{figure:surgery-link-example}
\end{figure}

\begin{remark}
  Conversely, a surgery presentation can be converted to a Heegaard
  splitting.  For instance, Lu's method in \cite{Lu:1992-1} tells us
  how to obtain a Heegaard splitting from a surgery link, as a product
  of explicit Dehn twists on an explicit surface.  By rewriting those
  Dehn twists in terms of the Lickorish twists, for instance by
  following the arguments of existing proofs that Lickorish twists
  generate the mapping class group (e.g, see
  \cite{Lickorish:1962-1,Lickorish:1964-1} or
  \cite{Farb-Margalit:2012-1}), one would obtain a word in the
  Lickorish twists which represents the mapping class, and in turn an
  upper bound for the Heegaard-Lickorish complexity of the 3-manifold.
  We do not address details here.
\end{remark}

\begin{remark}
  \label{remark:quadratic-bound}
  Theorem~\ref{theorem:surgery-and-HL-complexity} and (the proof of)
  Theorem~\ref{theorem:simplicial-and-surgery-complexity} immediately
  give a triangulation from a Heegaard splitting, together with the
  following complexity estimate:
  \[
  \csimp(M) \le 72\cdot (2\cdot \cHL(M)^2 + 3\cdot \cHL(M)).
  \]
  It tells us that the simplicial complexity is bounded by a
  \emph{quadratic} function in the Heegaard-Lickorish complexity.  A
  quadratic bound seems to be the best possible result \emph{from this
    method} (unless one finds a clever simplification of the
  resulting surgery link).  For instance, by generalizing the
  rightmost 5 components in Figure~\ref{figure:surgery-link-example}
  and considering the corresponding mapping class, one sees that there
  is actually a genus one mapping class of Lickorish word length $\le
  \ell$ for which the associated link $L$ has crossing number $\ge
  \frac{\ell}{2}(\frac{\ell}{2}-1)$.  In general, except for
  sufficiently small values of $\cHL$, this quadratic bound is weaker
  than the \emph{linear} bound in
  Theorem~\ref{theorem:simplicial-and-HL-complexity}\@.
\end{remark}

\begin{remark}  
  \label{remark:other-bounds-from-heegaard}
  The upper bound to the (pseudo-)simplicial complexity in terms of
  Heegaard splittings given in
  Theorem~\ref{theorem:simplicial-and-HL-complexity} is often stronger
  than Matveev's upper bound in~\cite{Matveev:1990-1, Matveev:2007-1}.
  We recall Matveev's result: suppose $M$ admits a Heegaard splitting
  $M=H_1\cup_\Sigma H_2$ with handlebodies $H_1$ and $H_2$ and
  Heegaard surface~$\Sigma$.  Let $\alpha$ and $\beta$ be the union of
  the meridian curves of $H_1$ and $H_2$ on $\Sigma$, respectively.
  Suppose $\alpha$ and $\beta$ are transverse,
  $n=\#(\alpha\cap \beta)$, and the closure of a component of
  $\Sigma\setminus (\alpha\cup\beta)$ contains $m$ points in
  $\alpha\cap\beta$.  Then $c(M) \le n-m$
  \cite[Proposition~3]{Matveev:1990-1},
  \cite[Proposition~2.1.8]{Matveev:2007-1}.  As an explicit example,
  let $\tau$, $\sigma$ be the $+1$ Dehn twists along the meridian and
  preferred longitude on the boundary of the standard solid torus in
  $S^3$, and consider the lens space $L$ with Heegaard splitting
  determined by the mapping class of $\sigma^k\tau^k$.  It is
  straightforward to see that $n=k^2+1$ and $m=4$ for this Heegaard
  splitting, so that the result in~\cite{Matveev:1990-1,
    Matveev:2007-1} gives $c(L) \le k^2-3$, a \emph{quadratic} upper
  bound.  On the other hand,
  Theorem~\ref{theorem:simplicial-and-HL-complexity} gives a
  \emph{linear} upper bound $c(L)\le\csimp(L) \le 1104k$, since
  $\sigma^k\tau^k$ has Lickorish word length $\le 2k$.  In fact, for
  arbitrary $N>0$, we can construct examples of lens spaces, using
  mapping classes of the form $(\sigma^k\tau^k)^N$ and
  $\tau^k(\sigma^k\tau^k)^N$, for which Matveev's upper bound
  $c(M) \le n-m$ has order $N$ (i.e., asymptotic growth of $k^N$)
  while Theorem~\ref{theorem:simplicial-and-HL-complexity} gives a
  linear upper bound.
\end{remark}

The rest of this section is devoted to the proof of
Theorem~\ref{theorem:simplicial-and-HL-complexity}\@.  The key idea
used in our proof below, which enables us to produce a more efficient
triangulation (cf.\ Remark~\ref{remark:quadratic-bound}), is that we
view Lickorish's surgery link
(Figure~\ref{figure:surgery-link-example}) as a link in the thickened
Heegaard surface.

\begin{proof}[Proof of
  Theorem~\ref{theorem:simplicial-and-HL-complexity}]
  Here we will prove the following statement, which is slightly
  sharper than Theorem~\ref{theorem:simplicial-and-HL-complexity}: if
  a closed $3$-manifold $M\ne S^3$ has Heegaard-Lickorish
  complexity~$\ell$, then the simplicial complexity of $M$ is not
  greater than $552\ell-120$.

  Suppose $h\in \Mod(\Sigma_g)$ gives a Heegaard splitting of a given
  $3$-manifold $M$, and suppose $h$ is a product of $\ell$ Lickorish
  generators.  Both $g$ and $\ell$ are nonzero, since $M\ne S^3$.  By
  Lickorish~\cite{Lickorish:1962-1}, $M$ is obtained by surgery on an
  $\ell$-component link $L$ in $S^3$, where each component has either
  $(+1)$ or $(-1)$-framing.  His proof tells us more about~$L$
  (another useful reference for this is \cite[Chapter~9,
  Section~I]{Rolfsen:1976-1})\@.  In fact, $L$ lies in a bicollar
  $\Sigma_g\times[-1,1]$ of $\Sigma_g$ in $S^3$, and each component is
  of the form $\alpha_i\times \{t\}$, $\beta_i\times \{t\}$, or
  $\gamma_i\times \{t\}$ for some $i$ and $t\in [-1,1]$.  An example
  is shown in Figure~\ref{figure:surgery-link-example}.  Let
  \[
    D=\Big(\bigcup_{i=1}^g \alpha_i\Big) \cup \Big(\bigcup_{i=1}^g
    \beta_i\Big) \cup \Big(\bigcup_{i=1}^{g-1} \gamma_i\Big).
  \]
  Then $L$ lies on $D\times[-1,1] \subset S^3$.



  
  Note that for a link in the bicollar $\Sigma_g\times[-1,1]$, if each
  component is regular with respect to the projection of
  $\Sigma_g\times[-1,1] \to \Sigma_g$, then the \emph{blackboard
    framing with respect to $\Sigma_g$} is well-defined; the preferred
  parallel with respect to the blackboard framing is defined to be the
  push-off along the $[-1,1]$ direction.  In particular, for our
  surgery link $L$, the blackboard framing with respect to
  $\Sigma_g$ is equal to the zero framing in~$S^3$.


  Now, in order to construct a triangulation of
  $\Sigma_g\times[-1,1] \setminus \nu(L)$, we proceed similarly to the
  proof of
  Lemma~\ref{lemma:triangulation-for-blackboard-framing-surgery}; the
  difference is that we now use a ``diagram'' on $\Sigma_g$, instead
  of a planar link diagram.  Let $G_0$ be the dual graph of $D$
  on~$\Sigma_g$.  Let $G$ be the graph shown in
  Figure~\ref{figure:dual-graph-with-added-parallels}, which is
  obtained by adding parallel edges to~$G_0$.

  \begin{figure}[H]
    \labellist
    \small
    \pinlabel{(top view)} at 171 131
    \pinlabel{(bottom view)} at 171 5
    \pinlabel{$D$} at 90 200
    \pinlabel{$G$} at 80 162
    \pinlabel{$D$} at 90 73
    \pinlabel{$G$} at 80 100
    \endlabellist
    \includegraphics{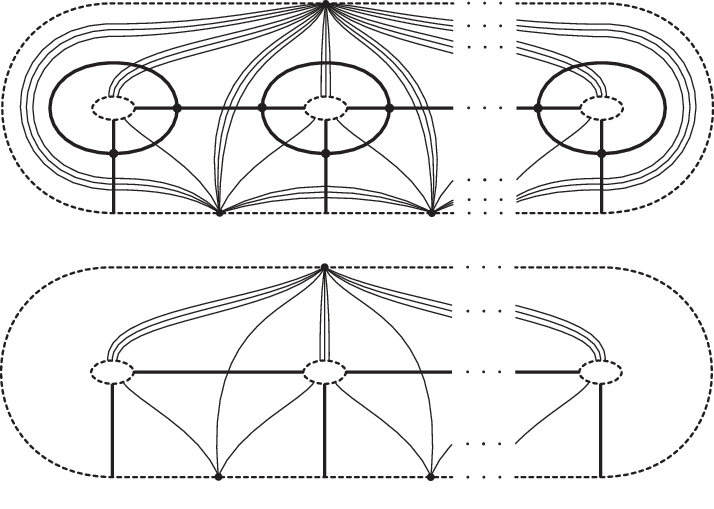}
    \caption{The graphs $D$ and $G$ on $\Sigma_g$, which are depicted
      in thick and thin edges respectively.}
    \label{figure:dual-graph-with-added-parallels}
  \end{figure}
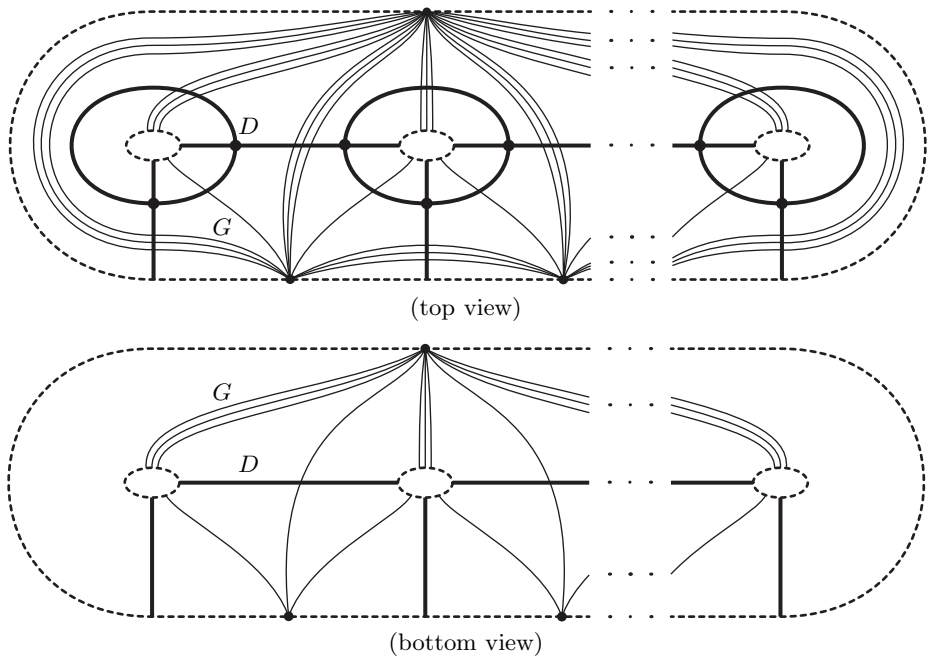

  Note that for each of the curves $\alpha_i$, $\beta_i$ and
  $\gamma_i$, an edge of $G_0$ dual to the curve is chosen and two
  parallels of the chosen edge are added to produce~$G$.  Each region
  of $G$ is a quadrangle or a bigon.  (Each quadrangle/bigon has no
  two edges which are identified, while vertices are allowed to be
  identified; using this, it can be verified that our construction
  described below gives a simplicial complex structure in which each
  tetrahedron has no identified vertices and is uniquely determined by
  its vertices.)

  Cutting $\Sigma_g\times[-1,1]\setminus \nu(L)$ along
  $G\times[-1,1]$, we obtain pieces corresponding to quadrangle
  regions and bigon regions; call them type (i) and (ii) respectively.
  See the left of Figure~\ref{figure:surgery-link-decomposition}.
  Cutting along $D\times[-1,1]$, a type (i) piece is divided into four
  cubic subpieces, and a type (ii) piece is divided into two
  triangular prism subpieces.  See the middle of
  Figure~\ref{figure:surgery-link-decomposition}.  Hatched quadrangles
  represent~$\partial\nu(L)$.

  \begin{figure}[ht]
    \labellist
    \pinlabel{$=$ a cone of} at 235 250
    \pinlabel{$=$ a cone of} at 235 70
    \small
    \pinlabel{type (i)} at 90 160
    \pinlabel{type (ii)} at 90 15
    \endlabellist
    \includegraphics{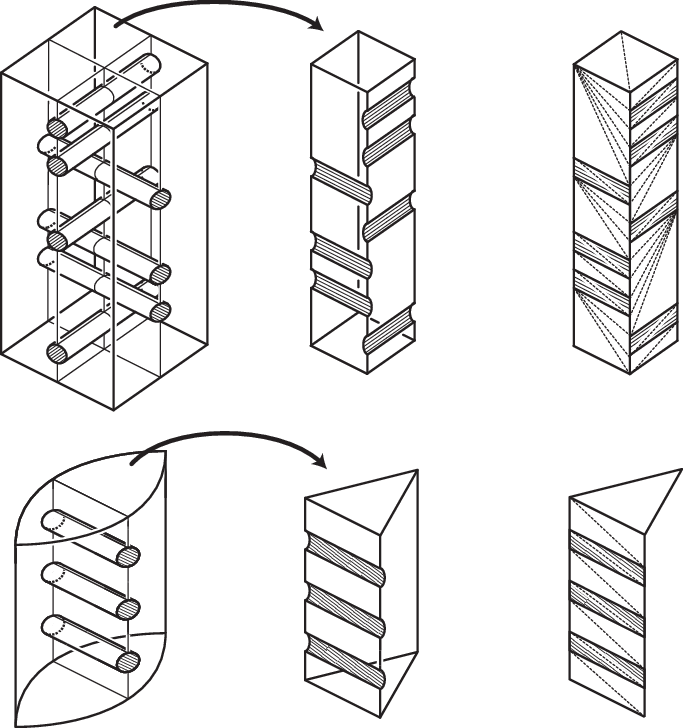}
    \caption{Decomposition of the surgery link exterior pieces.}
    \label{figure:surgery-link-decomposition}
  \end{figure}
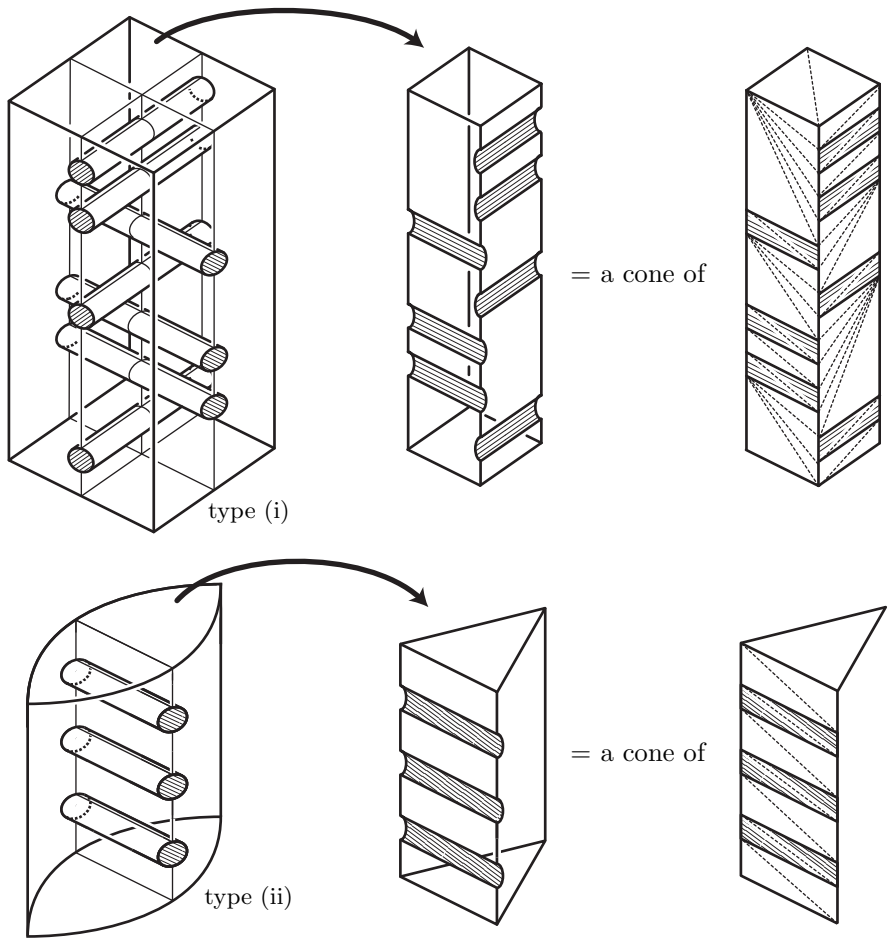

  For a type (i) subpiece, triangulate the three front faces of each
  subpiece as in the top right of
  Figure~\ref{figure:surgery-link-decomposition}, and then triangulate
  the subpiece by taking a cone at the opposite vertex, as we did in
  the proof of
  Lemma~\ref{lemma:triangulation-for-blackboard-framing-surgery}.  We
  claim that there are $6k+6$ tetrahedra in this subpiece
  triangulation, where $k$ is the number of hatched quadrangles in the
  subpiece.  The number of tetrahedra in the subpiece is equal to the
  number of triangles in the three front faces.  There are two
  triangles in the top face.  To count triangles in the remaining two
  faces, observe that the front middle vertical edge is divided into
  $2k+1$ $1$-simplices.  There are $4k+2$ triangles that have one of
  these $1$-simplices as an edge, and there are $2k+2$ remaining
  triangles.  Therefore there are total $6k+6$ triangles, as we
  claimed.

  A type (ii) subpiece is triangulated similarly, as depicted in the
  bottom of Figure~\ref{figure:surgery-link-decomposition}.  When a
  type (ii) subpiece has $k$ hatched quadrangles, its triangulation
  has $4k+3$ tetrahedra.

  Combining the triangulations of the subpieces, we obtain a
  triangulation of $\Sigma_g\times[-1,1] \setminus \nu(L)$.  To
  estimate the number of tetrahedra, first observe that the graph $D$
  has $3g-2$ vertices, where $g$ is the genus of the Heegaard
  surface~$\Sigma_g$.  Therefore its dual graph $G_0$ has $3g-2$
  quadrangular regions.  Since $2(3g-1)$ parallel edges have been added
  to $G_0$ and each of them introduces a bigon region, the graph $G$
  has $3g-2$ quadrangular regions and $6g-2$ bigon regions.  It
  follows that there are $12g-8$ type (i) subpieces and $12g-4$ type
  (ii) subpieces in $\Sigma_g\times[-1,1] \setminus \nu(L)$.  Also,
  observe that each component of $L$ passes through type (i)
  pieces at most three times, and type (ii) pieces twice.
  Therefore a component can contribute at most $4\cdot 3=12$ hatched
  quadrangles in type (i) subpieces, and $2\cdot 2=4$ hatched
  quadrangles in type (ii) subpieces.  It follows that there are at
  most
  \[
  6\cdot 12\ell + 6\cdot (12g-8) + 4\cdot 4\ell+3\cdot(12g-4) = 88\ell+108g-60
  \]
  tetrahedra in our triangulation of
  $\Sigma_g\times[-1,1] \setminus \nu(L)$.

  For later use, note that our triangulation restricted to
  $\Sigma_g\times\{t\}$ ($t=\pm1$) has $2(12g-8)+(12g-4)=36g-20$
  triangles, since the top face of each of the $12g-8$ type (i)
  subpieces consists of two triangles, and the top of each of the
  $12g-4$ type (ii) subpieces is a single triangle.

  Now we triangulate the inner and outer handlebodies, which are the
  components of $S^3\setminus (\Sigma_g\times(-1,1))$.  First we
  consider the outer handlebody.  Choose disjoint disks
  $D_0,D_1,\ldots,D_g$ in the outer handlebody such that
  $\partial D_i=\alpha_i$ for $i=1,\ldots,g$, and $\partial D_0$ is
  the union of the outermost edges of the graph $G$ in the top view of
  Figure~\ref{figure:dual-graph-with-added-parallels}; $\partial D_0$
  is parallel to the outer dotted circle in
  Figure~\ref{figure:dual-graph-with-added-parallels}.  Our
  triangulation on $\Sigma_g\times\{1\}$ divides $\partial D_0$ into
  $2g$ edges, each of $\partial D_1$ and $\partial D_g$ into 6 edges,
  and each $\partial D_i$ ($i=2,\ldots,g-1$) into 8 edges.  Extending
  this triangulation of the boundary, we triangulate $D_0$ into
  $2g-2$ triangles, each of $D_1$, $D_g$ into 4 triangles, and each
  $D_i$ ($i=2,\ldots,g-1$) into 6 triangles, by drawing edges joining
  vertices.  Cutting the outer handlebody along the disks $D_0$,
  $\ldots$, $D_g$, we obtain two $3$-balls $B_1$ and $B_2$.  Our
  triangulations of the $D_i$ and $\Sigma_g\times\{1\}$ give
  triangulations of $\partial B_1$ and~$\partial B_2$.  Triangulate
  each of $B_1$ and $B_2$ by taking the cone of the boundary.  Note
  that a triangle in $\Sigma\times\{1\}$ contributes one tetrahedron
  to $B_1\cup B_2$, while a triangle in $D_i$ contributes two
  tetrahedra to $B_1\cup B_2$.  It follows that the outer handlebody
  has at most $(36g-20) + 2\cdot(2g-2+6g-4) = 52g-32$ tetrahedra.

  For the inner handlebody, choose disjoint disks
  $D'_1,\ldots,D'_g, D''_1,\ldots,D''_{g-1}$ in the inner handlebody
  such that $\partial D'_i = \beta_i$ and $\partial D''_i=\gamma_i$.
  Similarly to the case of the disks $D_i$ above, our triangulation
  extends to $(\bigcup D'_i) \cup (\bigcup D''_j)$ where $D'_i$ and
  $D''_i$ are decomposed to 2 and 4 triangles respectively.  Cutting
  the inner handlebody along the disks $D'_i$ and $D''_i$, we obtain
  $g$ 3-balls.  Triangulate each 3-ball by taking the cone of the
  boundary.  A counting argument similar to the above shows that the
  inner handlebody has $(36g-20)+2(2g+4(g-1)) = 48g-28$ tetrahedra.
  

  To obtain the surgery manifold, attach and triangulate Dehn filling
  tori as in the proof of
  Lemma~\ref{lemma:triangulation-for-blackboard-framing-surgery}.
  Recall that the blackboard framing is equal to the zero framing in
  the present case.  Since each component of $L$ passes through two
  type (ii) pieces each of which introduces a half twist with respect
  to the blackboard framing, each Dehn filling torus can be assumed to
  be attached along the given $(\pm1)$-framing of $L$, by
  appropriately choosing diagonal edges used to triangulate the
  hatched quadrangles of type (ii) pieces in
  Figure~\ref{figure:surgery-link-decomposition}.  Therefore the
  surgery manifold is equal to the given~$M$.  Since there are at most
  $16\ell$ hatched quadrangles and each hatched quadrangle contributes
  a triangular prism which consists of $3$ tetrahedra in the Dehn
  filling tori, there are at most $48\ell$ tetrahedra in the Dehn
  filling tori.

  It follows that our triangulation of the surgery manifold $M$ has at
  most
  \[
  (88\ell+108g-60) + (52g-32) + (48g-28) + 48\ell = 136\ell+208g-120
  \]
  tetrahedra.  By
  Lemma~\ref{lemma:heegaard-genus-bound-from-Heegaard-Lickorish-complexity},
  we may assume that $g\le 2\ell$.  It follows that the simplicial
  complexity of $M$ is at most~$552\ell-120$.
\end{proof}

\section{Theorems~\ref{theorem:simplicial-and-HL-complexity}
  and~\ref{theorem:simplicial-and-surgery-complexity} are
  asymptotically optimal}
\label{section:linear-bounds-are-optimal}

In this section we prove
Theorem~\ref{theorem:linear-bounds-are-optimal} and related results.
For this purpose we use some results in~\cite{Cha:2014-1}.  First, we
need the following lower bound of the simplicial complexity.
In~\cite{Cheeger-Gromov:1985-1}, Cheeger and Gromov introduced the von
Neumann $L^2$ $\rho$-invariant $\rhot(M,\phi)\in \R$ which is defined
for a smooth closed $(4k-1)$-manifold $M$ and a homomorphism
$\phi\colon \pi_1(M)\to G$.  By deep analytic arguments, they showed
that for each $M$ there is a universal bound for the values of
$\rho(M,\phi)$~\cite{Cheeger-Gromov:1985-1}; that is, there is $C_M>0$
satisfying that $|\rhot(M,\phi)|\le C_M$ for any~$\phi$.
In~\cite{Cha:2014-1}, a topological approach to the universal bound
for $\rhot(M,\phi)$ is presented, and in particular, an explicit
linear universal bound is given in terms of the simplicial complexity
of 3-manifolds:

\begin{theorem}[{\cite[Theorem~1.5]{Cha:2014-1}}]
  \label{theorem:simplicial-complexity-unversal-bound}
  Suppose $M$ is a closed 3-manifold.  Then
  \[
  |\rhot(M,\phi)| \le 363090 \cdot \csimp(M)
  \]
  for any homomorphism~$\phi$.
\end{theorem}

In this paper, we will use the Cheeger-Gromov $\rho$-invariant as a
lower bound of the simplicial complexity.

For the lens space $L(n,1)$ and the identity map
$\id\colon \pi_1(L(n,1)) \to \Z_n$ ($n>0$), Lemma~7.1
of~\cite{Cha:2014-1} gives the following value of the Cheeger-Gromov
$\rho$-invariant, using the computation of
Atiyah-Patodi-Singer~\cite[p.~412]{Atiyah-Patodi-Singer:1975-2}:
\[
\displaystyle \rhot(L(n,1)),\id) = \frac{n}{3} + \frac{2}{3n} - 1.
\]
From this and
Theorem~\ref{theorem:simplicial-complexity-unversal-bound}, a lower
bound of the simplicial complexity of $L(n,1)$ is obtained.  We state
it as a lemma:

\begin{lemma}
  \label{lemma:lens-space-complexity}
  $\displaystyle \csimp(L(n,1)) \ge \frac{n-3}{1089270}.$
\end{lemma}

We remark that a pseudo-simplicial complexity analogue is given in
\cite[Corollary~1.15]{Cha:2014-1}.

Now we are ready to proof
Theorem~\ref{theorem:linear-bounds-are-optimal}\@.  In fact, the
following stronger inequalities hold, and
Theorem~\ref{theorem:linear-bounds-are-optimal} follows immediately
from them.

\begin{theorem}
  \label{theorem:lens-space-inequalities}
  \begin{gather*}
    \frac{1}{1089720} \cdot \Big( 1-\frac{3}{n} \Big) \cdot
    \cHL(L(n,1)) \le \csimp(L(n,1)),
    \\
    \frac{1}{1089720} \cdot \Big( 1-\frac{3}{n} \Big) \cdot
    \csurg(L(n,1)) \le \csimp(L(n,1)).    
  \end{gather*}
\end{theorem}

\begin{proof}
  Since $L(n,1)$ is obtained by the $n$-framed surgery on the unknot,
  it is easily seen that $\cHL(M)$, $\csurg(M) \le n$.  The desired
  inequalities follow from this and
  Lemma~\ref{lemma:lens-space-complexity}.
\end{proof}

In what follows we discuss a generalization and a specialization of
the lens space case we considered in
Theorem~\ref{theorem:lens-space-inequalities}.

First, the second inequality in
Theorem~\ref{theorem:lens-space-inequalities} generalizes for a larger
class of 3-manifolds.  For a knot $K$ in $S^3$, let $M(K,n)$ be the
3-manifold obtained by $n$-framed surgery on~$K$.  Let $g_4(K)$ be the
(topological) slice genus of~$K$.

\begin{theorem}
  \label{theorem:surgery-manifold-inequality}
  For any $n\ne 0$, 
  \[
  \frac{1}{1089720} \cdot \Big( 1-\frac{3+6g_4(K)}{|n|} \Big) \cdot
  \big(\csurg(M(K,n))-2c(K)\big) \le \csimp(M(K,n)).
  \]
\end{theorem}

\begin{proof}
  Let $\phi\colon \pi_1(M(K,n)) \to \Z_{|n|}$ be the abelianization.
  Due to~\cite[Equation (2.8)]{Cha:2015-2},
  \[
  |\rhot(M(K,n),\phi)| \ge \frac{1}{3} \cdot (|n|-3-6g_4(K)).
  \]
  By Theorem~\ref{theorem:simplicial-complexity-unversal-bound}, it
  follows that
  \begin{equation}
    \label{equation:csimp-knot-surgery}
    \csimp(M(K,n)) \ge \frac{1}{1089270} \cdot (|n|-3-6g_4(K)).
  \end{equation}
  By definition, $\csurg(M(K,n)) \le 2c(K) + |n|$.  From this and
  \eqref{equation:csimp-knot-surgery}, the desired inequality follows.
\end{proof}

On the other hand, if we consider the special case of lens spaces
$L(2k,1)$, then the inequalities in
Theorem~\ref{theorem:lens-space-inequalities} (and hence those in
Theorem~\ref{theorem:linear-bounds-are-optimal}) can be improved
significantly as follows.

\begin{theorem}
  For $k>1$, the following hold:
  \begin{gather*}
    \Big( 1-\frac{3}{2k\mathstrut} \Big) \cdot \cHL(L(2k,1)) \le
    \csimp(L(2k,1)),
    \\
    \Big( 1-\frac{3}{2k} \Big) \cdot \csurg(L(2k,1)) \le
    \csimp(L(2k,1)).
  \end{gather*}
\end{theorem}

\begin{proof}
  Due to Jaco, Rubinstein, and Tillman
  \cite{Jaco-Rubinstein-Tillman:2009-1}, the pseudo-simplicial
  complexity of $L(2k,1)$ is equal to $2k-3$ for $k>1$, and
  consequently $\csimp(L(2k,1)) \ge 2k-3$.  Using this in place of
  Lemma~\ref{lemma:lens-space-complexity} in the proof of
  Theorem~\ref{theorem:lens-space-inequalities}, we obtain the
  inequalities.
\end{proof}

We finish this section with a proof of
Theorem~\ref{theorem:asymptotic-growth}.

\begin{proof}[Proof of Theorem~\ref{theorem:asymptotic-growth}]
  Recall the definition of the ``largest possible value'' of the
  simplicial complexity for Heegaard-Lickorish complexity $\le \ell$:
  \[
  \sHL(\ell) := \sup\{\csimp(M) \mid \cHL(M)\le \ell\}.
  \]
  The first assertion of Theorem~\ref{theorem:asymptotic-growth},
  which says $\sHL(\ell)\in O(\ell)\cap \Omega(\ell)$, follows
  immediately from the estimate
  \begin{equation}
    \label{equation:sHL-estimate}
    \frac{1}{1089270} \le \limsup_{\ell \to \infty}
    \frac{\sHL(\ell)}{\ell} \le 552.
  \end{equation}
  which we prove in what follows.

  Fix~$\ell$.  For any $M$ with $\cHL(M) \le \ell$, we have
  \[
  \frac{\csimp(M)}{\ell} \le \frac{\csimp(M)}{\cHL(M)} \le 552
  \]
  by Theorem~\ref{theorem:simplicial-and-HL-complexity}\@.  Taking the
  supremum over all such $M$, we obtain $\sHL(\ell)/\ell \le 552$.
  From this we obtain the upper bound
  in~\eqref{equation:sHL-estimate}.

  By the definition of $\sHL(\ell)$, we have
  \[
  \frac{\csimp(M)}{\cHL(M)} \le \frac{\sHL(\cHL(M))}{\cHL(M)}
  \]
  for any~$M$.  By
  Theorem~\ref{theorem:lens-space-inequalities}, the
  limit supremum of the left hand side as $\cHL(M)\to\infty$ is
  bounded from below by $1/1089270$.  From this the lower bound in
  \eqref{equation:sHL-estimate} follows.

  The analogous statement for the function $\ssurg(k)$ is proved by
  the same argument.
\end{proof}

\bibliographystyle{amsalpha}
\renewcommand{\MR}[1]{}
\bibliography{research}

\end{document}